\documentclass[12pt]{article}

\usepackage{amssymb}
\usepackage{amsmath}
\usepackage{amsthm}
\newtheorem{thm}{Theorem}[section]
\newtheorem{lem}[thm]{Lemma}
\newtheorem{cor}[thm]{Corollary}
\newtheorem{prop}[thm]{Proposition}
\usepackage{graphicx,psfrag}
\usepackage{array,latexsym}
\usepackage{graphicx}
\usepackage{threeparttable}
\usepackage{booktabs}

\begin{document}

\title{The IC-indices of Some Complete Multipartite Graphs}
\author{ Chin-Lin Shiue \thanks{Department of Applied Mathematics, Chung Yuan Christian University,
Chung Li, Taiwan 32023. Research supported in part by NSC-99-2115-M033-001.
}
\and
Hui-Chuan Lu    \thanks {{\tt  Corresponding author}; Center for Basic Required Courses, National United University, Maioli, Taiwan 36003.
{\tt  E-mail:hjlu@nuu.edu.tw}.  Research supported in part by  NSC-102-2115-M-239-002.  }
}

\maketitle
\begin{abstract}
A coloring of a connected graph $G$ is a function $f$ mapping the vertex set of $G$ into the set of all integers.
For any subgraph $H$ of $G$, we denote the sum of the values of $f$ on the vertices of $H$ as $f(H)$.
If for any integer $k\in \{1,2,\cdots,f(G)\}$, there exists an induced connected subgraph $H$ of $G$
such that $f(H) = k$, then the coloring $f$ is called an IC-coloring of $G$.
The IC-index of $G$, denoted as $M(G)$, is the maximum value of $f(G)$ over all possible IC-colorings $f$ of $G$.
In this paper, we present  a useful method from which a lower bound 
on the IC-index of any complete multipartite graph can be derived.
Subsequently, we show that, for $m\geq 2 ~\mbox{and} ~n\geq 2$, our lower bound on $M(K_{1(n),m})$ is the exact value of it. 


\end{abstract}


\section{Introduction}
\label{intro}

 Given a connected simple graph $G$, a \textit{coloring} of $G$ is a function $f$ mapping $V(G)$ into $\Bbb N $.
       For any subgraph $H$ of $G$, we denote the sum $\sum_{v\in V(H)}f(v)$ as $f(H)$. If for any integer $k\in \{1,2,3,\cdots,f(G)\}$, there exists an induced connected subgraph $H$ of $G$ such that $f(H) = k$, then the coloring $f$ is called an
       \textit{IC-coloring} of $G$ .
       Every connected graph $G $ admits a trivial IC-coloring which assigns the value 1 to every vertex of $G$. 
The highest possible value of $f(G)$ is referred to as the \textit{IC-index} of a graph $G$, denoted as $M(G)$, that is,
      $$ M(G) = \max\{f(G) ~|~ f ~\mbox{is an IC-coloring of} ~G\}.$$
       An IC-coloring $f$ satisfying $f(G) = M(G)$ is called a \textit{maximal IC-coloring}
       of $G$. In this paper, we only consider simple graphs. For the terminologies and
       notations in graph theory, please refer to \cite{west}.

The problem of IC-coloring of finite graphs originated from the postage stamp problem in number theory, which has been studied in some literature \cite{AB,Guy,HL,LW,Lun,Mos,SLK,Sel,SKM,St}.      
In 1992, Glenn Chappel formulated the IC-coloring problem as a "subgraph sums problem" and  showed  that $M(C_n)\leq n^2-n+1$. Later, in 1995,  Penrice \cite{Pen} introduced the IC-coloring
       as the stamp covering and showed that $M(K_n)=2^n-1$ for $n\geq 1$ and $M(K_{1,n})=2^n+2$ for $n\geq 2$. In 2005, Salehi et. al.\cite{SLK} proved that $M(K_{2,n})=3\cdot 2^n+1$ for $n\geq 2$. Along with the result by Shiue and Fu \cite{SF}   who showed that $M(K_{m,n})=3\cdot 2^{m+n-2}-2^{m-2}+2$, for $2\leq m \leq n$, in 2008, the problem regarding complete bipartite graphs was completely settled. In this present paper, we deal with \textit{complete multipartite graphs}. A complete multipartite graph $K_{m_1,m_2,\cdots, m_k}$ is a graph whose vertex set can be partitioned into $k$ partite sets $V_1,V_2,\cdots, V_k$, where $|V_i|=m_i$ for all $i \in \{1, 2, \cdots, k\}$, such that there are no edges within each $V_i$ and any two vertices from different partite sets are adjacent. A complete multipartite graph with $k$ partite sets is called a \textit{complete $k$-partite graph} as well. We also denote as $K_{1(n),m_{n+1}, m_{n+2}, \cdots, m_k}$, $n\le k$, the complete $k$-partite graphs in which there are $n$ partite sets which are of size one and the rest $(k-n)$ partite sets have sizes $m_{n+1},m_{n+2},\cdots,$ and $m_{k}$. Therefore $K_{1(n)}$ represents the complete graph $K_{n}$.

In this paper,  we introduce some useful lemmas in Section 2. In Section 3, our main results 
are presented. We start with a useful proposition which gives a lower bound on the IC-index of the join of an independent set and a given connected graph. Consequently, a lower bound on the IC-index of any complete mutlipartite graph can be deduced. We shall show that our lower bound on $M(K_{1(n),m})$ is indeed the exact value of it for $m\geq 2 ~\mbox{and} ~n\geq 2$.
A concluding remark will be given in Section 4.

\section{Preliminaries}
\label{sec:2}

       Some basic know results from  \cite{SF} are introduced in this section. They are very useful in the discussion of our main results. For brevity, we let $[1,\ell]$, $\ell \in \Bbb N$, denote the set $\{1,2,\cdots,\ell\}$. A sequence consisting of $0$'s and $1$'s is called a \textit{binary sequence}.
        \noindent
        \begin{lem}\cite{SF}\label{lemma 2.1}
        If $a_{1}, a_{2},\cdots, a_{n}$ are $n$ positive integers which satisfy that
        $a_{1}=1$ and $a_{i}\leq a_{i+1}\leq \sum_{j=1}^{i}a_{j}+1$ for all $i \in [1,n-1].$
        Then, for each $\ell\in [1,\sum_{j=1}^{n}a_{j}]$, there exists a binary sequence $c_{1}, c_{2},\cdots, c_{n}$
        such that $\ell=\sum_{j=1}^{n}c_{j}a_{j}$.
        \end{lem}

        \noindent
        \begin{lem}\cite{SF}
        If $s_{0},s_{1},\cdots,s_{n}$ is a sequence of integers, then for each $i\in [1,n]$ there exists an integer $r_{i}\in \mathbb{Z}$ such that
        $s_{i}=\sum_{j=0}^{i-1}s_{j}+r_{i}$ and $\sum_{j=0}^{n}s_{j}=
        2^{n}s_{0}+\sum_{j=1}^{n}2^{n-j} r_{j}.  $
        \end{lem}

        \noindent
        \begin{lem}\cite{SF}
        Let $V(G)=\{u_{1},u_{2},\cdots,u_{n}\}$. If $f$ is an IC-coloring of G such that $f(u_{i})\leq f(u_{i+1})$ ~for all $i \in [1,n-1]$,
        then $f(u_{1})=1$ and $f(u_{i+1})\leq \sum_{j=1}^{i}f(u_{j})+1$ for all $i \in [1,n-1]$.
        \end{lem}

        \noindent
        \begin{lem}\cite{SF}
        Let f be an IC-coloring of a graph G such that $f(u_{i}) < f(u_{i+1})$ ~for $i \in [1,n-1],$
        where $V(G)=\{u_{1},u_{2},\cdots,u_{n}\}.$ For each pair  $(i_{1},i_{2})$ where $1\leq i_{1} < i_{2}\leq n, $
        if $f(u_{i_{1}})=\sum_{j=1}^{i_{1}-1}f(u_{j})+1$ and $u_{i_{1}}u_{i_{2}}\notin E(G)$, then either
        $f(u_{i_{2}})\leq \sum_{j=1}^{i_{2}-1}f(u_{j})-f(u_{i_{1}})$ or $f(u_{i_{2}+1})\leq f(u_{i_{1}})+f(u_{i_{2}})$.
        \end{lem}

        \noindent
        \begin{lem}\cite{SF}
        Let $r_{1},r_{2},\cdots,r_{n}$ be $n$ numbers. If there are two integers $i$ and $k$ such that $1\leq i< k\leq n$
        and $r_{i} < r_{k},$ then
        \[\begin{array}{clr}
        \sum_{j=1}^{n}2^{n-j}r_{j} < \sum_{j=1}^{n}2^{n-j}r_{j}-(2^{n-i}r_{i}+2^{n-k}r_{k})+(2^{n-i}r_{k}+2^{n-k}r_{i}).
        \end{array}\]
        \end{lem}

        \noindent
        \begin{lem}\cite{SF}
        Let $f$ be an IC-coloring of a graph $G$. If $G$ has $\ell$ induced connected
        subgraphs and there are $2k$ distinct induced connected subgraphs $H_{1},G_{1},H_{2},G_{2},
        \cdots,H_{k}$, $G_{k}$ of $G$ such that
        $f(H_{i})=f(G_{i})$ for all $i \in [1,k],$ then $f(G)\leq \ell-k.$
        \end{lem}

\section{Main Result}
\label{sec:3}
We start this section with a useful method for deriving a meaningful lower bound on the IC-index of the join of an independent set and a given graph. Subsequently, 
we show that the lower bound on $M(K_{1(n),m})$ derived from our method, for $m\geq 2 ~\mbox{and} ~n\geq 2$, also serves as an upper bound on it. This determines the exact value of $M(K_{1(n),m})$.

\subsection{Lower Bounds on the IC-indices of Complete Multipartite Graphs}

For the derivation of  lower bounds, we view complete multipartite graphs as being generated by a graph operation starting with graphs with some vertices and no edges. The \textit{join} of two disjoint graphs $H_0$ and $H_1$, written $H_0\vee H_1$, is the graph with vertex set $V=V(H_0)\cup V(H_1)$ and edge set
     \[\begin{array}{clr}
        E=E(H_0)\cup E(H_1)\cup \{(u,v)~|~u\in V(H_0),~v\in V(H_1)\}.
     \end{array}\]
Let $O_{m}$ be the graph with $m$ vertices and no edges, then the join of $O_1$ and $K_{n}$, or $O_1\vee K_{n}$, is the complete graph $K_{n+1}$ and  the graph $O_{m}\vee K_{n}$ is exactly the complete multipartite graph $K_{1(n),m}$. Observe that the join of $O_{m}$ and 
$O_{n}$ is the complete bipartite graph $K_{m,n}$.  The join of $O_{m}$ and $K_{n_1,n_2}$ forms the complete tripartite graph $K_{n_1,n_2,m}$. Since joining $O_{m}$ with a complete $(k-1)$-partite graph generates a  complete $k$-partite graph, we are concerned about how the value of the IC-index of a graph changes as we joining $O_{m}$ to that graph.


 \noindent
\begin{prop}
 If $g$ is an IC-coloring of a connected graph $G$, then there exists an IC-coloring $f$ of $O_{m}\vee G$ such that 
$f(O_{m}\vee G)= 2^{m}g(G)+1$ for $m\geq 1$.
\end{prop}

\begin{proof}
Let $V_0=V(O_m)=\{w_1,w_2,\cdots,w_m \}$ and $V_1=V(G)=\{v_1,v_2,\cdots,v_n \}$. 
We define $f$ on $V(O_{m}\vee G)$ as $f(v_{i})=g(v_{i}) $ for $i\in[1,n]$, $f(w_{j})=2^{j-1}g(G)$ for
$j \in [1,m-1]$ and $f(w_{m})=2^{m-1}g(G)+1$. The value of
$f(O_{m}\vee G)$ can be calculated as follows.

\[\begin{array}{clr}
f(G)&=g(G)+\sum_{j=1}^{m} 2^{j-1}g(G)+1\\
        &=g(G)+(2^{m}-1)g(G)+1\\
        &= 2^{m}g(G)+1.
\end{array}\]\\
Given any integer $k\in [1,2^{m}g(G)+1]$, we need to identify an
induced connected subgraph $H$ of  $O_{m}\vee G$ such that $f(H)=k$.
Since $g$ is an IC-coloring of $G$ and $f(v)=g(v)$ for each vertex in $G$, 
the desired induced connected subgraph exists for each $k\in [1,g(G)]$.  
For $k\in [g(G)+1,2^{m-1}g(G)]$, we rewrite $k$ into the form
$k=qg(G)+r$ where $1\leq r \leq g(G)$.  Hence, $1\leq q \leq 2^{m-1}-1$ and then there exists a binary sequence
$c_{1}, c_{2},\cdots, c_{m-1}$, which are not all zero, such that
$q=\sum_{i=1}^{m-1}c_{i}2^{i-1}$. It follows that
$qg(G)=\sum_{i=1}^{m-1}c_{i}2^{i-1}g(G)=\sum_{i=1}^{m-1}c_{i}f(w_{i})$.
 Now, let $H'$ be a connected subgraph of $G$ such that $f(H')=r$ and 
let $W=V(H')\cup\{w_{i} \mid c_{i}=1$, $i\in [1,m-1]\}$. Since $H'$ is connected, the subgraph
$H$ induced by $W$ is also connected and satisfies $f(H)=k$. Next, if $k=2^{m-1}g(G)+1$, 
then the subgraph $H$ induced by the single vertex $w_m$ fits our need.
Finally, for $k\in [2^{m-1}g(G)+2,2^{m}g(G)+1]$, we write $k$ as
 $k=(2^{m-1}g(G)+1)+k'$, where $1\leq k' \leq 2^{m-1}g(G)$. By the
above argument, there is a connected subgraph $H'$ such that
$f(H')=k'$ and $V(H')\cap V_1\neq\emptyset$. Let $W=V(H')\cup
\{w_m\}$, then the subgraph $H$ induced by $W$ is
connected and we have $f(H)=k$. The result follows.
\end{proof}

Using the known result $M(K_n)=2^n-1$ for $n\geq 1$ \cite{Pen}, lower bounds on the IC-indices of $K_{1(n),m}$  and $K_{1(n),m_1,m_2, \cdots,m_l}$ can be easily obtained by Proposition 3.1.

 \noindent
\begin{cor}
  $M(K_{1(n),m})= M(O_{m}\vee K_{n})\geq 2^{m+n}-2^{m}+1$ for $m\geq 1 ~\mbox{and} ~n\geq 1.$
  
\end{cor}

For $m=1$, this lower bound matches the known value of the IC-index of $K_{n+1}$. We will show that this inequality  is in fact an equality when $m\geq 2 ~\mbox{and}~n\geq 2$.  
 
 \noindent
\begin{cor}
For $m_1\geq m_2\geq\cdots \geq m_{\ell}$, 
\[\begin{array}{clr}
M(K_{1(n),m_{\ell},m_{\ell-1}, \cdots,m_2,m_1})&=M(O_{m_1}\vee O_{m_2}\vee \cdots \vee O_{m_\ell}\vee K_{n})\\
        &\geq(2^{m_1}(2^{m_2}(\cdots (2^{m_{\ell}}(2^n-1)+1)\cdots)+1)+1).
\end{array}\]
\end{cor}

Furthermore, since the IC-index of a bipartite graph is known, a lower bound on $M(K_{m_1,m_2,m_3})$ can be derived as well. By successively applying Proposition 3.1, a lower bound on the IC-index of any complete multipartite graph can be easily found.

\subsection{ The Exact Value of the IC-index of $K_{1(n),m}$}

Now, we consider the graph $K_{1(n),m}$, or $O_{m}\vee K_{n}$. We shall show that the IC-index of this graph is $2^{m+n}-2^{m}+1$.
In the remainder of the this paper, we let $G=K_{1(n),m}=O_{m}\vee K_{n}$, $V_0=V(O_m)$ and $V_1=V(K_n)$. 
We introduce some properties of the graph $K_{1(n),m}$ and any maximal IC-coloring of it first.

\noindent
\begin{lem}
$K_{1(n),m}$ has $~(2^{m+n}-2^{m}+m)$ induced connected subgraphs.
\end{lem}

\begin{proof}
Any induced connected subgraph $H$ of $G$ must
satisfy exactly one of the following three conditions: (i)$V(H)\subseteq V_1$ and $V(H)\neq\emptyset$;
(ii)$V(H)\subseteq V_0$ and $|V(H)|=1$; (iii)$V(H)\cap V_1\neq\emptyset$ and $V(H)\cap V_0\neq\emptyset$.
Therefore, the number of induced connected graphs of $G$ is $(2^{n}-1)+m+(2^{m}-1)(2^{n}-1) = 2^{m+n}-2^{m}+m.$
\end{proof}

\noindent
\begin{prop}
If $f$ is a maximal IC-coloring of $K_{1(n),m}$, then $f(u)
\neq f(v)$ for each pair of distinct vertices $u$ and $v$ in $V(K_{1(n),m})$.
\end{prop}

\begin{proof}
Suppose that there are two distinct vertices $u$ and $v$ in $V(G)$ such that $f(u) =f(v)$.
For subsets of vertices $V_0' \subseteq V_0 \setminus\{u,v\}$ and $V_1'\subseteq V_1 \setminus\{u,v\}$
, we denote as $H_{u}$ and $H_{v}$ the subgraphs of $G$ induced by $V_0'\cup V_1'\cup\{u\}$ and $V_0'\cup V_1'\cup\{v\}$ respectively. Then, $H_u\neq H_v$ and $f(H_{u})=f(H_{v})$.
Let $S$ be the set of all possible pairs $(V_0',V_1')$ such that $H_u$ and  $H_v$ are both connected, and let $p=|S|$.  One can see from Lemma 2.6 and Lemma 3.4 that $f(G)\leq(2^{m+n}-2^{m}+m)-p$.

 Observe that any connected subgraph $H$ of $G$ must satisfy exactly one of the two conditions: (i)$V(H)\cap V_1\neq\emptyset$, or (ii)$V(H)\cap V_1=\emptyset$ and $|V(H)\cap V_0|=1$. Base on this fact, we now evaluate the number $p$ in the following three possible cases. \\

\noindent
{Case (a)}: if $u,v\in V_1$, then $p=2^{|V_1\setminus\{u,v\}|}\cdot 2^{|V_0|}=2^{m+n-2}$. \\

\noindent
{Case (b)}: if $u,v\in V_0$, then either $V_1'\neq\emptyset$ or $V_0'=V_1'=\emptyset$. Thus we have $p=2^{|V_0\setminus\{u,v\}|}\cdot (2^{|V_1|}-1)+1=2^{m-2}(2^{n}-1)+1=2^{m+n-2}-2^{m-2}+1$. \\

\noindent
{Case (c)}: if exactly one of $u$ and $v$ is in $V_0$, then either $V_1'\neq\emptyset$ or $V_0'=V_1'=\emptyset$. It follows that $p=2^{m-1}(2^{n-1}-1)+1=2^{m+n-2}-2^{m-1}+1$.\\

The value of $p$ in Case (c) is the minimum among these three cases. This leads to an upper bound on $f(G)$ as follows.
\[\begin{array}{clr}
f(G)&\leq(2^{m+n}-2^{m}+m)-p\\
    &\leq(2^{m+n}-2^{m}+m)-(2^{m+n-2}-2^{m-1}+1)\\
    &=2^{m+n}-2^{m+n-2}-2^{m-1}+m-1.
\end{array}\]\\
Since $m,n\geq 2$, we have $2^{m+n-2}\geq 2^m$ and $-2^{m-1}+m-1<0$. This implies that $f(G)<2^{m+n}-2^{m}+1$ which is a contradiction to Corollary 3.2 and we have the result.
\end{proof}

Now we are in a position to prove that the lower bound given in Corollary 3.2 also serves as an upper bound on $M(K_{1(n),m})$.
However, the proof is too involved that we need some more notations to facilitate the whole discussion process. In what
follows, for the given maximal IC-coloring $f$ of $G$, we always assume that $\{u_1,u_2,\cdots,u_{m+n} \}$ is the vertex set of $G$ such that $f(u_i)<f(u_{i+1})$ for all $i \in[1,m+n-1]$. Thus, $f(u_1)=1$ and  $f(u_2)=2$ are always true. Besides, we also define $f_0=0$ and denote the sum $\sum_{j=1}^{i}f(u_{j})$ as $f_{i}$ for $i\in[1,m+n]$. Let us introduce some useful facts first.

\noindent
\begin{lem}
If $f$ is a maximal IC-coloring of $K_{1(n),m}$, then \\
(1) $f_j\leq  2^{j-i}(f_i+1)-1$ for every pair $(i,j)$ such that $1\leq i < j\leq m+n$.\\
(2) $f_i\geq 3\cdot 2^{i-2}$ for each $i\in [2,m+n]$.
\end{lem}

\begin{proof}
(1) Given $i \in [1,m+n-1]$, we let $s_0=f_i$ and $s_k=f(u_{i+k})$ for each $k\in [1,m+n-i]$.
By Lemma 2.3, we have $s_k\leq f_{i+k-1}+1=\sum_{\ell=0}^{k-1}s_{\ell}+1$.
For each $j \in [i+1,m+n]$, one can see from Lemma 2.2 that
\[\begin{array}{cll}
f_j
   &=&\sum_{\ell=0}^{j-i}s_\ell\\
   &\leq& 2^{j-i}f_i+\sum_{\ell=1}^{j-i}2^{j-i-\ell}\cdot 1\\
   &=& 2^{j-i}(f_i+1)-1.
 \end{array}\]
(2) Suppose that there exists some $i\in [2,m+n]$ such that $f_i\leq 3\cdot 2^{i-2}-1$. According to part(1), we have\\
\[\begin{array}{cll}
f(G)&=&f_{m+n}\\
   &\leq&2^{m+n-i}(f_i+1)-1\\
   &\leq&2^{m+n-i}(3\cdot 2^{i-2})-1\\
   &=&2^{m+n}-2^{m+n-2}-1\\
   &<&2^{m+n}-2^{m}+1.
 \end{array}\]

\noindent
This contradicts to Corollary 3.2 and we have the result.
\end{proof}

\noindent
\begin{lem}
Suppose that $f$ is a maximal IC-coloring of $K_{1(n),m}$. Let $s_{0}=f_{0}=0$ and $s_i=f(u_i)$ for $i\in [1,m+n]$.
If each $r_i $ is the integer such that $s_i=\sum_{j=0}^{i-1}s_{j}+r_i$, $i\in [1,m+n]$, then
$r_i\leq 1$. Furthermore, if $r_i\leq 0$ for all $i\in \{j \mid u_j \in V_0\setminus\{ u_{k_0}\}\}$, where 
$k_0=\max \left.\{j\mid u_{j} \in V_0 \right.\}$,  then $f(K_{1(n),m})\leq 2^{m+n}-2^{m}+1$.
\end{lem}

\begin{proof}
The first result is trivial from Lemma 2.3. To prove that second result, 
we describe the IC-coloring $\bar{f}$ defined in the proof of Proposition 3.1 for $O_{m}\vee K_{n}$ explicitly. 
Let $V_0=V(O_m)=\{w_1,w_2,\cdots,w_m \}$ and $V_1=V(K_n)=\{v_1,v_2,\cdots,v_n \}$. 
Since we are given a maximal IC-coloring $g$ of $K_n$ defined as $g(v_i)=2^{i-1}$ for each $v_i\in V_1$ [7], we have an IC-coloring $\bar{f}$ of $O_{m}\vee K_{n}$ defined as 
$\bar{f}(v_{i})=g(v_i)=2^{i-1} $ for $i\in[1,n]$, $\bar{f}(w_{j})=2^{j-1}(2^{n}-1)$ for $j \in [1,m-1]$ and
$\bar{f}(w_{m})=2^{m-1}(2^{n}-1)+1$. Then $\bar{f}(O_{m}\vee K_{n})=2^{m+n}-2^{m}+1$. 
Let us rearrange the vertices of $G$ into a new order $\{ \bar{u}_1,\bar{u}_2,\cdots,\bar{u}_{m+n}\}$  such that $\bar{f}(\bar{u}_i)<\bar{f}(\bar{u}_{i+1})$ for all $i\in [1,m+n-1]$. Then $\bar{f}(\bar{u}_i)=\sum_{j=1}^{i-1}\bar{f}(\bar{u}_j)+\bar{r}_i$, where $\bar{r}_i=1$ for $i\in[1,n] \cup \{n+m\}$, and  $\bar{r}_i=0$ for $i\in[n+1,n+m-1]$. According to the assumption in this Lemma, $r_i\leq 0$ for all $i\in \{j~|~ u_j \in V_0\setminus\{ u_{k_0}\}\}$. 
If $r_{m+n}=1$, then among $r_1,r_2, \cdots,r_{m+n}$ there are at most $(n+1)$ of them taking the value one and at least $(m-1)$ of them having the values no more than zero. By comparing the distributions of the $1$'s in $\bar{r}_i$'s and ${r}_i$'s, we conclude from Lemma 2.5 that $f(G)$ does not exceed the value $\bar{f}(G)=2^{m+n}-2^{m}+1$. If $r_{m+n}=0$, then there are at most $n$ $r_i$'s being one and at least $m$ $r_i$'s not exceeding zero. Lemma 2.5 again guarantees the truth of the inequality $f(G)\leq \bar{f}(G)$ in this case.
\end{proof}

Next, we consider the case where the assumption in Lemma 3.7 is violated, namely, there is a $r_i$ having the value one for some $i \in \{ j\mid u_j \in V_0\setminus \{u_{k_0}\}\}$.

\noindent
\begin{lem}
Let $f$ be a maximal IC-coloring of $K_{1(n),m}$. Let $s_{0}=f_{0}=0$, $s_i=f(u_i)$  and $r_i=s_i-\sum_{j=0}^{i-1}s_{j}$ for $i\in [1,m+n]$. Suppose that $S_1=\{ i < k_0 \mid r_i=1 \textnormal{ and } u_i \in V_0 \} \neq \emptyset$, where $k_0=\max \left.\{j\mid u_{j} \in V_0 \right.\}$. Let $i_1=  \min \left. S_1\right.$, $t=|\{ u_i \in V_0 \mid i<i_1\}|$ and $S_2=\{ i \geq i_1+1 \mid u_i \in V_0 \textnormal{ and } f(u_i)> f_{i-1}-f(u_{i_1})\}$.
\begin{itemize}
  \item[(i)] If $S_2=\emptyset$, then  $f(K_{1(n),m})\leq 2^{m+n}-2^{m}+1$.
  \item[(ii)] If $S_2 \neq \emptyset$, then  $f(K_{1(n),m})\leq 2^{m+n}-2^{m+n-2}(1-2^{-(i_2-i_1)})-(2^t-1)(3\cdot 2^{m+n-i_1-1}+2^{m+n-i_2-1})$, where $i_2=  \min \left.S_2 \right.$.
  \item[(iii)] If $S_2 \neq \emptyset$ and $n\geq 3$, then  $f(K_{1(n),m})\leq 2^{m+n}-2^{m}+1$.
\end{itemize}
\end{lem}

\begin{proof}
Observe that, among $r_1,r_2, \cdots,r_{i_1-1}$, there are at most $(i_1-t-1)$ of them having the value one. By Lemma 2.2 and Lemma 2.5 we have that
\[\begin{array}{clr}
        f_{i_1-1}&= 2^{i_1-1}s_0+\sum_{j=1}^{i_1-1}2^{i_1-1-j}r_j\qquad\qquad\qquad\quad \\
                 &\leq 2^{i_1-1}\cdot 0+\sum_{j=1}^{i_1-t-1}2^{i_1-1-j}\cdot 1 \qquad\qquad\qquad\qquad\qquad\qquad\qquad 
\qquad\qquad\qquad\qquad\qquad\\
                 &=2^{i_1-1}-2^{t}\dotfill $(1)$
  \end{array}\]
  \noindent
        and
\[\begin{array}{clr}
         f_{i_1}&=f_{i_1-1}+f(u_{i_1})\qquad\qquad\qquad\quad \\
                &=2f_{i_1-1}+1 \qquad\qquad\quad \qquad\quad\qquad\qquad\qquad \qquad\qquad \qquad\qquad \quad\qquad \qquad\qquad\qquad\qquad \\
                &\leq 2^{i_1}-2^{t+1}+1.\dotfill $(2)$
         \end{array}\]\\
(i) Consider the case $s'_{0}=f_{i_1}$ and $s'_{j}=f(u_{i_{1}+j})$  for  $j\in [1, m+n-i_1]$. Let $r'_{j}= s'_{j}-\sum_{\ell=0}^{j-1}s'_{\ell}$, then $r'_j\leq 1$.  
Furthermore, if $S_2=\emptyset$, that is, $f(u_{i_1+j})\leq f_{{i_1}+j-1}-f(u_{i_1})$ for all $u_{i_1+j} \in V_0$, then $f(u_{i_1+j})-\sum_{i=1}^{{i_1}+j-1}f(u_i) \leq -f(u_{i_1})\leq -1$. Since $r'_j\leq -1$ for all $j \in \{j\leq m+n-i_1 \mid u_{i_1+j} \in V_0\}$ , there are at most $(n-(i_1-t-1))$ 1's in the values of $r'_1,r'_2, \cdots,r'_{m+n-i_1}$.  By Lemma 2.2, 2.5 and Inequality (2), we have

       \[\begin{array}{clr}
        f(G)&=\sum_{j=0}^{m+n-i_1}s'_{j}\\
            &=2^{m+n-i_{1}}f_{i_1}+\sum_{j=1}^{m+n-i_1}2^{m+n-i_1-j}r'_{j}\\
            &\leq 2^{m+n-i_{1}}(2^{i_1}-2^{t+1}+1)+\sum_{j=1}^{n-(i_1-t-1)}2^{m+n-i_1-j}\cdot 1+\\
            &\quad \sum_{j=n-(i_1-t-1)+1}^{m+n-i_1}2^{m+n-i_1-j}(-1)\\
            &\leq 2^{m+n-i_{1}}(2^{i_1}-2^{t+1}+1)+2^{m+n-i_1}-2^{m-t-1}-(2^{m-t-1}-1)\\
            &\leq 2^{m+n}+2^{m+n-i_1+1}-2^{m+n-i_1+t+1}-2^{m-t}+1.
       \end{array}\]\\
Note that $n\geq i_1-t-1\geq 0$ and $t\geq 0$. If $n=i_1-t-1$ or $t=0$, then $f(G)\leq 2^{m+n}-2^{m}+1$. Otherwise, $n\geq i_1-t$ and $t\geq 1$, then  $2^{m+n-i_1+1}-2^{m+n-i_1+t+1}\leq 2^{m+n-i_1+t}-2^{m+n-i_1+t+1}= -2^{m+n-i_1+t}\leq -2^m$.
              This implies that
                 $f(G) \leq 2^{m+n}-2^{m}-2^{m-t}+1<2^{m+n}-2^{m}+1$ and the result in (i) is asserted.\\
(ii) If $S_2 \neq \emptyset$, then $f(u_{i_2})> f_{{i_2}-1}-f(u_{i_1})$. By Lemma 2.4 we have that $f(u_{i_2+1}) \leq f(u_{i_1})+ f(u_{i_2})$ and then
 \[\begin{array}{cll}
        f_{i_{2}+1}&=f_{i_{2}-1}+f(u_{i_{2}})+f(u_{i_{2}+1})\\
                   &\leq 
                   f_{i_{2}-1}+2f(u_{i_{2}})+f(u_{i_{1}})\\
                   &\leq f_{i_{2}-1}+2(f_{i_{2}-1}+1)+(f_{i_{1}-1}+1)\quad \mbox{ }\mbox{ }\qquad\qquad \qquad\qquad\qquad\qquad\qquad \qquad\qquad\qquad\\
                   &=3f_{i_{2}-1}+f_{i_{1}-1}+3.\dotfill $(3)$
        \end{array}\]
Now, the desired upper bound on $f(G)$ can be derived as follows.
\[\begin{array}{cll}
                f(G)&=f_{m+n} \leq 2^{m+n-(i_{2}+1)}(f_{i_2+1}+1)\qquad \qquad\qquad\qquad\qquad\qquad \mbox{ }\mbox{ }({\textnormal{by Lemma 3.6(1)}})\\
                 &\leq 2^{m+n-(i_{2}+1)}[(3f_{i_{2}-1}+f_{i_{1}-1}+3)+1]\qquad\qquad\qquad\qquad\qquad \mbox{ }({\textnormal{by Inequality (3)}})\\
                 &\leq 2^{m+n-(i_{2}+1)}[3(2^{i_{2}-i_{1}}(f_{i_{1}-1}+1)-1)+f_{i_{1}-1}+4]\qquad \qquad\quad \mbox{ }({\textnormal{by Lemma 3.6(1)}})   \\
                 &= 2^{m+n-(i_{2}+1)}[(3\cdot 2^{i_{2}-i_{1}}+1)(f_{i_{1}-1}+1)]\\
                 &= 2^{m+n-(i_{2}+1)}(3\cdot 2^{i_{2}-i_{1}}+1)(2^{i_{1}-1}-(2^t-1)) \qquad \quad \qquad\qquad \mbox{ }({\textnormal{by Inequality (1)}})\\
                 &= 3\cdot2^{m+n-2}+2^{m+n-(i_{2}-i_{1})-2}-(2^t-1)(3\cdot 2^{m+n-i_{1}-1}+2^{m+n-i_2-1})\\
                 &= 2^{m+n}-2^{m+n-2}(1-2^{-(i_2-i_1)})-(2^t-1)(3\cdot 2^{m+n-i_1-1}+2^{m+n-i_2-1}) \dotfill $(4)$ \\
        \end{array}\]\\
(iii) Since  $i_{2}-i_{1} \geq 1$, $2^{m+n-2}(1-2^{-(i_2-i_1)})\geq 2^ {m+n-3}$. It follows that $f(G) \leq 2^ {m+n}-2^ {m+n-3} < 2^{m+n}-2^{m}+1$ when $n\geq 3$. The proof is completed.

\end{proof}

We have shown that $f(G)\leq 2^{m+n}-2^{m}+1$ is valid in many cases. With some further discussion, the final conclusion can be achieved.

\noindent
\begin{thm}
$M(K_{1(n),m})= 2^{m+n}-2^{m}+1$.
\end{thm}

\begin{proof}
Let $f$ be a maximal IC-coloring of $G$. We adopt the notation used in Lemma 3.8. By Lemma 3.7 and Lemma 3.8, it suffices to show that $f(G)\leq 2^{m+n}-2^{m}+1$ holds when  $S_2 \neq \emptyset$ and $n=2$. First note that when $n=2$, we have $0\leq i_1-t-1\leq 2$ and the upper bound on $f(G)$ in (4) can be rewritten as
\[\begin{array}{clr}
        f(G)&\leq
        2^{m+2}-2^m(1-2^{-(i_2-i_1)}))-(2^t-1)(3 \cdot 2^{m-i_{1}+1}+2^{m-i_2+1})\\
        &=  (2^{m+2}-2^{m})+2^{m-(i_2-i_1)}-3\cdot2^{m-i_{1}+t+1}(1-2^{-t})+2^{m-i_2+1}(2^t-1) \qquad\qquad \quad\qquad \\
        &=  2^{m+2}-2^{m}-3\cdot2^{m-(i_{1}-t-1)}(1-2^{-t})+2^{m-(i_2-i_1)}(1-2^{-i_{1}+1}(2^t-1)).\dotfill $(5)$\\
        \end{array}\] \\
(1)  If $t\geq 2$, then
      \[ \begin{array}{cll}
        f(G) & \leq & (2^{m+2}-2^{m})-3\cdot2^{m-2}(1-\frac{1}{4})+2^{m-1}\\
             & < & 2^{m+2}-2^{m}+1.
 \end{array}\] \\
(2) If $t=1$ and $i_1\leq 3$, then
\[\begin{array}{cll}

 f(G) 
  & \leq & (2^{m+2}-2^{m})-3\cdot2^{m-1}(1-2^{-1})+2^{m-1}(1-2^{-2}(2-1))\\
 & = & (2^{m+2}-2^{m})+(3/4-3/2)\cdot 2^{m-1}\\
  & < & 2^{m+2}-2^{m}+1.
  \end{array}\]\\
(3) We have so far consider all but the following four cases: $(i_1,t)=(1,0)$,$(2,0)$, $(3,0)$,$(4,1)$.
Let us have a closer investigation of the value $f(G)$ again before literally starting the discussion of these cases.
Now, consider the the situation where $s_{0}=f_0$, $s_{i}=f(u_i)$ and $s_{i}=\sum_{j=0}^{i-1}s_{j}+r_{i}$ for $i\in [1, i_{2}-1]$. From the definition of $i_1$ and $i_2$, we see that $r_i \leq 0$ for all $i\in \{i\leq i_2-1 \mid u_i \in V_0\}\setminus\{i_1\}$. Since $n=2$,  among $r_1,r_2,\cdots,r_{i_2-1}$, there are at most three of them taking the value one.  We therefore have
\[\begin{array}{cll}
f_{i_{2}-1} & \leq & 2^{i_2-1}s_{0}+\sum_{j=1}^{i_2-1}2^{i_2-1-j}\cdot r_j\\
 & \leq & 2^{i_{2}-2}+2^{i_{2}-3}+2^{i_{2}-4}\\
 & = & 7\cdot 2^{i_{2}-4}.
 \end{array}\]

Making use of this fact and Inequality (1), an upper bound on $f_{i_{2}+1}$ can be derived from Inequality (3) as follows.
\[\begin{array}{cll}
             f_{i_{2}+1}&\leq 3f_{i_{2}-1}+f_{i_{1}-1}+3 \\
                       & \leq 21\cdot2^{i_{2}-4}+2^{i_1-1}-2^{t}+3.
 \end{array}\]
This leads to an upper bound on $f(G)$.
\[\begin{array}{cll}
                f(G)& < 2^{m+2-(i_{2}+1)}(f_{i_2+1}+1) \mbox{ }\qquad\qquad\qquad\qquad \qquad\qquad \qquad\qquad({\textnormal{by Lemma 3.6(1)}})\\
                 &= 2^{m+2-(i_{2}+1)}(21\cdot2^{i_{2}-4}+2^{i_{1}-1}-2^{t}+4)\\
                   &=(21\cdot 2^{m-3})+ [2^{m-(i_{2}-i_1)}- 2^{m+1-i_{2}+t}+2^{m+3-i_{2}}]\\
                    &=(2^{m+2}-2^m-3\cdot2^{m-3})+[2^{m-(i_2-i_1)}(1-2^{-i_1+t+1}+2^{-i_1+3})]\dotfill $(6)$\\
  \end{array}\]

Now, let us have the discussion for the remaining four cases.\\

\noindent
 {\bf Case 1.} $(i_1,t)=(1,0)$.

In this case, $f(u_{i_1})=f(u_1)=1$. If $i_2\geq 5$, then Inequality (6) gives \\
\[\begin{array}{cll}
 f(G)
 &< &(2^{m+2}-2^{m})-3\cdot2^{m-3}+2^{m-4}(1-2^{0}+2^{2})\\
  & < & 2^{m+2}-2^{m}+1.
   \end{array}\]

When $i_2=4$, since $f(u_2)=2>f_1-f(u_1)$ and $f(u_3)> f(u_2)=f_2-f(u_1)$, we have $\{u_2,u_3\}\subseteq V_1$ by the definition of $i_2$. So, $V_1=\{u_2,u_3\}$ and  $V_{0}=\{u_i \mid i \in \{1\}\cup[4,m+n] \}$. Now, $f(u_3)\leq f_2+1=4$ and $f_3\leq 7$. 
Since $i_2=4$, we know that $f(u_4)>f_3-f(u_1)$ holds.
Lemma 2.4 guarantees that $m+n\geq 5$ and $f(u_5)\leq f(u_1)+f(u_4)\leq 9$, giving $f_5=f_3+f(u_4)+f(u_5)\leq f_3+(f_3+1)+f(u_5)\leq 7+8+9\leq 24$.
Now, suppose that $m+n\geq 6$, then the fact $f_6 \geq 3\cdot 2^{6-2}=48$ from Lemma 3.6(2) implies that $f(u_6) \geq 48-f_5 \geq 24$, which means $f(u_6) >f_5-f(u_1)$. One can then obtain from Lemma 2.4 that $m+n\geq 7$ and $f(u_7)\leq f(u_1)+f(u_6)\leq 1+(f_5+1)=26$. However, this leads to
 $f_{7}= f_5+f(u_6)+f(u_7)  \leq 24 +25+26 = 75 < 3\cdot2^{7-2}$, contradicting to Lemma 3.6(2). We therefore conclude that the only possible situation is "$m+n=5$" and $f(G)=f_5\leq 24 < 2^{m+2} -2^m+1$.

When $i_2=3$, from the definition of $i_2$, we have $u_2\in  V_1$ and also $f(u_4)=f(u_{i_2+1})\leq f(u_{i_1})+f(u_{i_2})=f(u_1)+f(u_3)$ by Lemma 2,4.
Now, $2=f(u_2)<f(u_3)\leq f_2+1=4$. From Lemma 3.6(2), we see that $f_4\geq 3\cdot 2^{4-2}=12$ and thus $12\leq f_4=f_2+f(u_3)+f(u_4)\leq 3+f(u_3)+(f(u_1)+f(u_3))=2f(u_3)+4$ which implies that $f(u_3)=4$. It follows that $f_3=7$ and $f(u_4)\leq f(u_1)+f(u_3)= 5$ by Lemma 2.4 because $f(u_3)>f_2-f(u_1)$ and $u_1u_3 \notin E(G)$. Since $f(u_4)>f(u_3)=4$, we see that $f(u_4)=5$ and $f_4=12$.
Suppose that $m+n \geq 5$, then Lemma 3.6(2) gives $f_{5} \geq  3\cdot 2^{5-2}=24$. One can see that $f(u_5) =f_5-f_4 \geq 24-12=12>9=f(u_4)+f(u_3)$ and $f(u_4)>f_3-f(u_3)$. Lemma 2.4 then guarantees that $u_4$ must be in $V_1$. We therefore conclude that
$V_1=\{u_2,u_4\}$ and  $V_{0}=\{u_i \mid i \in \{1,3\}\cup[5,m+n] \}$.
Next, since  $f(u_{5})\geq 12>f_4-f(u_1)$, we have from Lemma 2.4 that $m+n\geq 6$ and $f(u_{6})\leq f(u_{1})+f(u_{5})\leq 1+(f_4+1)\leq 14$. However, this leads to $f_{6} \leq  f_4+f(u_5)+f(u_6)\leq 12+(12+1)+14 = 39< 3\cdot2^{6-2}$ which is a contradiction to Lemma 3.6(2). Therefore the only possible situation is "$m+n=4$" and then $f(G)=f_{4}=12\leq 2^{m+2}-2^m+1$.

If $i_2=2$, then $f(u_{i_1})=f(u_1)=1$ and $f(u_{i_2})=f(u_2)=2=f_1+1$. Since $f(u_3)=f(u_{i_2+1})\leq f(u_{i_1})+f(u_{i_2})=3$ and $f(u_3)>f(u_2)=2$, we have $f(u_3)=3$ and then $f(u_{3})> f_{2}-f(u_1)$. In addition, the fact $f_{4} \geq  3\cdot 2^{4-2}=12 $ implies that  $f(u_{4})=f_4-f_3 \geq 12-6=6>4 =f(u_3)+f(u_1)$. From Lemma 2.4, we see that $u_3$ must be in $V_1$. When $m=2$, $f(G)=f_{4} \leq  1+2+3+(f_3+1) = 13 = 2^{m+2}-2^m+1$.
Next, we consider the situation when $m\geq 3$.
Note that the inequality $f_5\geq 3\cdot 2^{5-2}=24$ gives $f(u_{5})=f_5-f_4 \geq 24-13=11>8 \geq f(u_1)+f(u_4)$. Along with the fact that $f(u_4)\geq 6> 5=f_3-f(u_1)$, one can see that $u_4$ must also be in $V_1$ from Lemma 2.4. Therefore, we have $V_1=\{u_3,u_4\}$ and $V_0=\{u_i~|~i \in \{1,2\}\cup [5,m+n]\}$. Suppose that $f(u_5) \geq 12$, then $f(u_5)>11\geq f_4-f(u_2)$. Lemma 2.4 guarantees that  $m+n\geq 6$ and $f(u_{6})\leq f(u_{2})+f(u_{5})\leq 2+(f_4+1)\leq 16$. However, these facts lead to $f_6=f_4+f(u_5)+f(u_6)\leq 13+14+16=43 < 3\cdot 2^{6-2}$, contradicting to Lemma 3.6(2). Hence, $f(u_5)=11$ is asserted. Now, suppose again that $m+n\geq 6$. The inequality $f_6\geq 3\cdot 2^{6-2}=48$ from Lemma 3.6(2) implies $f(u_{6})=f_6-f_5 \geq 48-24=24> f_5-f(u_1)$. So we have from Lemma 2.4 that $m+n\geq 7$ and $f(u_{7})\leq f(u_{1})+f(u_{6})\leq 1+(f_5+1)\leq 26$. However, this gives $f_{7}  \leq  f_5+f(u_6)+f(u_7)\leq f_5+(f_5+1)+26\leq  75 <  3\cdot2^{7-2}$, contradicting to Lemma 3.6(2) again. Therefore, "$m+n=5$" is the only possible situation and then $f(G) = f_5=24 <  2^{m+2}-2^{m}+1$. \\

\noindent
 {\bf Case 2.} $(i_1,t)=(2,0)$.

 In this case, $f(u_1)=1$, $f(u_2)=2$ and $u_1\in V_1$. If $i_2\geq 5$, then the upper bound on $f(G)$ in (6) can be rewritten as \\
$$f(G)\leq (2^{m+2}-2^{m})-3\cdot2^{m-3}+2^{m-3}(1-2^{-1}+2^{1}) <  2^{m+2}-2^{m}+1.$$

If $i_2= 4$, then since $f(u_{3})> 1= f_2-f(u_{2})$, one can see that $u_3$ must be in $V_1$ from the definition of $i_2$. Therefore we have $V_1=\{u_1,u_3\}$ and  $V_{0}=\{u_i \mid i \in \{2\}\cup[4,m+n] \}$. Note that $f(u_{4})\leq f_3+1=8$. From the definition of $i_2$ and  Lemma 2.4, we have  $f(u_{5})\leq f(u_2)+f(u_4) = 2+ f(u_4)$. Now, $3\cdot 2^{5-2} \leq f_5= f_3+f(u_4)+f(u_5)\leq (1+2+4)+f(u_4)+(2+f(u_4))$. This implies that $f(u_4)=8=f_3+1$ and $f(u_5)\leq 10$. Hence, $ f(u_5)>f(u_4)>7 \geq f_3=f_4-f(u_4)$.
It follows from Lemma 2.4 that $f(u_6)\leq f(u_4)+f(u_5)\leq 8+10=18$ and then $f_6=f_4+f(u_5)+f(u_6)=15+10+18=41< 3\cdot 2^{6-2}$ which contradicts to Lemma 3.6(2). We therefore conclude that $i_2=4$ can never occur in this case.

If $i_2=3$, then $u_1\in V_1$ and $f(u_{3})\leq f_2+1=4$. From the definition of $i_2$ and by virtue of Lemma 2.4, we have  $f(u_{4})\leq f(u_2)+f(u_3)\leq 6$. Suppose that $m+n\geq 5$. Since Lemma 3.6(2) gives that $f_4 \geq 3\cdot 2^{4-2}=12$ and $f_4\leq 1+2+f(u_3)+(f(u_2)+f(u_3))=5+2f(u_3)$, we obtain that $f(u_3)=4$. Now, the inequality $f_5\geq 3\cdot 2^{5-2}=24$ from Lemma 3.6(2) leads to $f(u_5)=f_5-f_4\geq 24-(1+2+4+6)=11>f(u_3)+f(u_4)$. One can be sure that $u_4$ must be in $V_1$ by Lemma 2.4 because $f(u_4)>f(u_3)=4>f_3-f(u_3)$ and $f(u_3)=f_2+1$. Hence, we have $V_1=\{u_1,u_4\}$ and $V_0=\{u_i~|~i \in \{2,3\}\cup [5,m+n]\}$. In addition, since $f(u_5)\geq 11> f_4-f(u_3)$, we know from Lemma 2.4 that  $m+n\geq 6$ and $f(u_6)\leq f(u_3)+f(u_5)\leq 4+(f_4+1)\leq 18$. These facts together lead to $f_{6}=f_4+f(u_5)+f(u_6)  \leq  13+14+18=45<3\cdot 2^{6-2}$ which contradicts to Lemma 3.6(2). We therefore conclude that the inequality $m+n\geq 5$ is impossible to be true in Case 3. So, $m+n=4$ and then $f(G)=f_4 \leq 13=2^{m+2}-2^m+1$.  \\

\noindent
{\bf Case 3.} $(i_1,t)=(3,0)$.\\

In this case, $V_1=\{u_1, u_2\}$ can be easily seen from the definition of $i_1$. So, $V_0=\{u_i|i \in [3,m+n]\}$. Now, we have $f(u_1)=1$, $f(u_2)=2$ and $f(u_3)= f_2+1=4$. Let $k=f(G)-(f(u_1)+f(u_2))$ and let $H$ be an induced subgraph of $G$
  such that $f(H)=k$. If $V_0\setminus V(H) \neq \emptyset$, we denote as $i_0$ the minimum element in $\{ i \mid u_i \in V_0\setminus V(H)\}$, then $i_0\geq 3$ and $V(H)\subseteq V_1\cup \{u_i \mid i\geq i_0+1\}$. Therefore,
$f(H)\leq f(u_1)+f(u_2)+\sum_{i=i_0+1}^{m+n}f(u_i)\leq f(u_1)+f(u_2)+\sum_{i=4}^{m+n}f(u_i) < \sum_{i=3}^{m+n}f(u_i)=f(H)$
, giving a contradiction. We then have that $V_0\subseteq V(H)$. Besides, since $\sum_{i=3}^{m+n}f(u_i)=k$, we conclude that $V(H)=V_0$ and $H$ is disconnected. Therefore, Case 3 is impossible to occur. \\

\noindent
{\bf Case 4.} $(i_1,t)=(4,1)$.\\

In this case, $V_1=\{u_1, u_2\}$ and $V_0=\{u_i \mid i \in [3,m+n]\}$ are also true.  Now, $f(u_1)=1$, $f(u_2)=2$ and $f(u_3)= f_2=3$ and  $f(u_4)= f_3+1= 7$. Let $k=f(G)-(f(u_1)+f(u_2)+$ $f(u_3))$. In exactly the same way as we used in Case 3, one can show that Case 4 is not possible either.   \\

Since $f(G)\leq 2^{m+2}-2^m+1$ is valid in all possible situations. The proof is completed.

\end{proof}


\section{Conclusion}
\label{sec:4}

In this work, we have given lower bounds on the IC-index of all complete multipartite graphs and  
have shown that our lower bound on $M(K_{1(n),m})$ is the exact value of it. 
Our coloring constructed in Proposition 3.1 is indeed a qualified maximal IC-coloring.

For further study of this problem, one can try to show that the lower bound given in Proposition 3.3 will also be an upper bound 
on $M(K_{1(n),m_{\ell},m_{\ell-1}, \cdots,m_2,m_1})$ for all $\ell\ge 2$. 
We conjecture that the inequality in Proposition 3.3 is in fact an equality in the case where $\ell \ge 2,~ m_{\ell}\geq 2 ~\mbox{and}~n\geq 2$.


\end{document}